\documentclass[a4paper]{amsart}

\usepackage[all]{xy}
\usepackage{amssymb}

\title{Rank varieties for Hopf algebras}

\author{Sarah Scherotzke}
\address{Mathematical Institute \\ University of Oxford \\ 24-29 St.\ Giles \\ Oxford OX1 3LB \\ United Kingdom}
\email{scherotz@maths.ox.ac.uk}

\author{Matthew Towers}
\address{Mathematical Institute \\ University of Oxford \\ 24-29 St.\ Giles \\ Oxford OX1 3LB \\ United Kingdom}
\email{towers@maths.ox.ac.uk}

\keywords{Rank varieties, Hopf algebras}
\subjclass{16W30, 16E40 (2000)}

\newtheorem{theo}{Theorem}[section]
\newtheorem{lem}[theo]{Lemma}
\newtheorem{cor}[theo]{Corollary}

\newtheorem{prop}[theo]{Proposition}
\newtheorem{rem}[theo]{Remark}

\newtheorem{defn}[theo]{Definition}

\newcommand{\Z}{\operatorname{\mathbb{Z}}\nolimits}
\newcommand{\ext}{\operatorname{Ext}}
\newcommand{\End}{\operatorname{End}}
\newcommand{\Hom}{\operatorname{Hom}}
\newcommand{\stend}{\operatorname{\underline{End}}}
\newcommand{\id}{\operatorname{\mathrm{id}}}
\newcommand{\res}{\operatorname{res}}

\newcommand{\Ann}{\operatorname{\mathrm{Ann}}}

\newcommand{\aut}{\operatorname{\mathrm{Aut}}}
\newcommand{\uqsl}{\ensuremath{ u_q (\mathfrak{sl}_2)}}

\begin{document}
\date{\today}
\begin{abstract}
We construct rank varieties for the Drinfel'd double of the Taft
algebra and for $\uqsl$.  For the Drinfel'd double
when $n=2$ this uses a result which identifies a family of
subalgebras that control projectivity of $\Lambda$-modules
whenever $\Lambda$ is a Hopf algebra satisfying a certain
homological condition.  In this case we show that our rank variety
is homeomorphic to the cohomological support variety. We also
show that $\ext^*(M,M)$ is finitely generated over the cohomology
ring of the Drinfel'd double for any finitely-generated module $M$.
\end{abstract}

\maketitle

\begin{section}{Introduction}
Rank varieties were first introduced by Carlson in \cite{carl:vars}, with the aim of providing easily computable homological information about finite-dimensional modules over the group algebra $kG$ of an elementary abelian $p$-group, where $k$ is a field of characteristic $p$.  Combined with Quillen's stratification theorem \cite[5.6]{bensonII} they became a useful tool for arbitrary finite groups, since the rank variety is much easier to work with than the more abstractly defined cohomological support variety.

Since Carlson's original definition rank varieties have been introduced in many other contexts, for example for
$p$-restricted Lie
algebras \cite{FP}, finite-dimensional cocommutative Hopf
algebras \cite{FPev},
quantum complete intersections \cite{BEH}, and for certain tensor products of Taft algebras \cite{PW}.

These definitions of rank variety all have the property that the rank
variety can be computed explicitly, usually by
determining if a module is projective when viewed as a module for
certain commutative subalgebras generated by an nilpotent element. Also rank varieties should characterize
projectivity of a module, that is, the rank variety is trivial if
and only if the module is projective. 

In this paper we define rank varieties for Drinfel'd doubles of the
Taft algebra and for $\uqsl$, a finite dimensional quotient of the quantised enveloping algebra of $\mathfrak{sl}_2$. These algebras are
Hopf algebras that are neither commutative nor cocommutative, and the latter is an important example of a small quantum group. The Drinfel'd double of a
finite-dimensional Hopf algebra was defined by Drinfel'd to provide
solutions to the quantum Yang-Baxter equation arising from
statistical mechanics. The Drinfel'd double of the Taft algebra also turns out to be of interest in knot theory \cite{RW}. The
representation theory of Drinfel'd doubles of Taft
algebras has been studied in \cite{EGST} and its cohomology ring, which turns out to be finitely generated, was computed by Taillefer in \cite{taillefer}.

This paper is organized as follows. In Section \ref{prelims} we summarize some properties of finite-dimensional Hopf
algebras that will be needed later.  Some basic
properties of the support variety are stated. In Section \ref{detecting_proj_section} we
develop a general condition for a set of subalgebras to detect
projectivity based on cohomological properties of a
finite-dimensional Hopf algebra. In the next two sections we work with the  Drinfel'd double of the Taft algebra when $n=2$ and then in general, along with $\uqsl$. 

First we show that the Drinfel'd double of the Taft algebra with
$n=2$ satisfies the cohomological properties needed to apply our
results of Section \ref{detecting_proj_section}. Therefore we can define rank varieties in
this case. The rank variety is defined via a set of
subalgebras generated by a nilpotent element and a unit. The rank
variety is given as union of lines through the origin of $k^2$. We
show that the support variety and the rank variety are
homeomorphic and that the rank variety satisfies the tensor
product property, that is the rank variety of the tensor product
of two modules is equal to the intersection of the respective rank
varieties.

Finally in Section \ref{n>2section} we introduce rank varieties for the
Drinfel'd doubles of the Taft algebra for any $n$ and for
$\uqsl$. In the general case we get the rank
variety using rank varieties defined for Morita equivalent blocks.
Unlike the case $n=2$ we cannot find subalgebras generated by a
nilpotent element in order to define rank varieties, but we use
the known structure of blocks to determine the basic algebra to
which the blocks are Morita equivalent. It then turns out that we
can use theory developed in \cite{BEH} to define rank varieties
for these algebras. In the case $n=2$ the two approaches used
define the same rank variety and for any $n$ this rank variety is
also given in terms of the projective space $P^1$.

In this paper all algebras are finite-dimensional algebras over a field $k$, and all modules are finite-dimensional left modules.

\section{Preliminaries} \label{prelims}
In this section we establish some notation and recall the definition of support varieties for finite-dimensional Hopf algebras.  We also prove two technical results that will be needed later on.

\begin{subsection}{Varieties}

Let $A$ be a finite-dimensional algebra over an algebraically closed field $k$ and let $A$-mod be the category of finite-dimensional left $A$-modules.  Suppose we have a map $V$ from $A$-mod to a set $S$ containing a distinguished element $0$.  The following four conditions are important properties of all the existing ``module variety'' theories:

\begin{enumerate}
\item[(C1)] $V(M\oplus N)=V(M)\cup V(N)$ for all $M, N \in A$-mod.

\item[(C2)] Let $0 \to M_1 \to M_2 \to M_3 \to 0$ be an exact sequence of
finite-dimensional $A$-modules, then $V(M_i)\subset V(M_j)\cup
V(M_s)$ where $\{i,j,s\}=\{1,2,3\}$.

\item[(C3)] $V(\Omega^n(M))=V(M)$ for all $n\in\Z$ and for all $M \in
A$-mod.

\item[(C4)] (Dade's Lemma) $V(M)=\{0\}$ if and only if $M$ is projective.
\end{enumerate}

If $A$ is a Hopf algebra, the cohomological support variety of a module can be defined as follows.  Denote by $\ext^{>0}_A(k,k)$ the unique homogeneous maximal ideal of $\ext^{ev}_A(k,k)$. Let $M$ and $N$ be two finitely generated $A$-modules and let $\psi_M:\ext_A^{ev}(k,k) \to \ext_A^{ev}(M,M)$ be the ring homomorphism given by tensoring an extension representing an element of $\ext_{A}^n(k,k)$ with $M$.  Using $\psi_M$ we can view $\ext^{*}_{A}(M,N)$ as a
$\ext^{ev}_{A}(k,k)$-module.

\begin{defn} Let $\Ann(M,N)$ denote the annihilator of $\ext_A^*(M,N)$ under the action of $\ext^{ev}_{A}(k,k)$. Let $\operatorname{Proj}\ext^{ev}_{A}(k,k)$ be the set of homogeneous prime ideals not contained in any homogeneous prime except $\ext_A^{>0}(k,k)$.  The {\textbf support variety} $V^s(M,N)$ is the set of elements of $\operatorname{Proj}\ext^{ev}_{A}(k,k)$ containing $\Ann(M,N)$.  We write $V^s(M)$ for $V^s(M,M)$.
\end{defn} 

Note that $V^s(M) \subset V^s(k)$. We assume for the rest of this section that $\ext_A^{ev}(k,k)$ is a finitely generated commutative ring and that $\ext_A^*(M,N)$ is finitely generated as a $\ext_A^{ev}(k,k)$-module. Then applying verbatim the proof of \cite[Proposition 2]{PW} gives the following results:
\begin{lem}
Let $M$ and $N$ be finite-dimensional $A$-modules. Then (C1)-(C4) hold and we have
\begin{enumerate}

\item $V^s(M,N) \subset V^s(M) \cap V^s(N)$.

\item $V^s(M\otimes N)\subset V^s(M) \cap V^s(N)$.
 
\item $V^s(M)= \bigcup_{S} V^s(M,S)$, where the union is over all simple $A$-modules $S$.
\end{enumerate}

\end{lem}
Following Carlson \cite{carl:vars} we can associate to a non-zero element
$\zeta \in \ext_A ^n(k,k) \cong \underline{\Hom}(\Omega^n(k), k)$ the
kernel of the corresponding map $\hat \zeta: \Omega^n(k) \to k$.
We denote this $A$-module by $L_{\zeta}$. Let  $\zeta$ be a homogeneous element in $\ext_A^{ev}(k,k)$, then
$\langle \zeta \rangle $ denotes the subvariety
consisting of all elements of $\operatorname{Proj}\ext^{ev}_{A}(k,k)$ containing $\zeta$. The proof of
\cite[Proposition 3]{PW} goes through to give
\begin{prop} \label{zeta}
For any finite-dimensional $A$-module $M$ we have
$V^s(M\otimes L_{\zeta})=V^s(M) \cap \langle \zeta \rangle$.
\end{prop}

\end{subsection}

\begin{subsection}{The Heller translate and reciprocity}

Let $\Lambda$ be a finite-dimensional Hopf algebra over a field $k$. Let $M$ be an indecomposable $\Lambda$-module which has period 2, so that there is an exact sequence
\begin{equation} \label{Mseq} 0 \rightarrow M \rightarrow X_1 \rightarrow X_0 \rightarrow M \rightarrow 0 \end{equation}
where the $X_i$ are projective.  All even-dimensional Ext-groups $\ext_\Lambda ^{2n} (M,M)$ are isomorphic because $\ext_\Lambda ^{m} (A,B) \cong \ext _\Lambda ^{m-1} (\Omega(A), B)$, and in fact they are all isomorphic to the stable endomorphism ring $\stend(M)$.  It follows that $\Omega^2$ induces an automorphism of $\stend(M)$.  Carlson gives a condition in \cite[Proposition 3.6]{carl:vars} for this automorphism to be the identity.
His proof was in the group algebra case, but it goes through for arbitrary Hopf algebras.

\begin{lem} \label{omegaf} Let $f \in \End(M)$ represent an element $\xi \in \ext _\Lambda ^2 (M,M)$ in the image of the map $- \otimes M : \ext_\Lambda ^2 (k,k) \rightarrow \ext _\Lambda ^2 (M,M)$.  Then $\Omega ^2 (\underline{f}) = \underline{f}$, where the underline denotes the image in the stable category. \end{lem}

\begin{proof}
Let $(\partial, F_*)$ be a $\Lambda$-projective resolution of $k$.  Then $(\partial \otimes \id, F_i \otimes M)$ is a projective resolution of $M$, and there is a chain map $\theta$  between this resolution and (\ref{Mseq}):

\[ \xymatrix{
\Omega ^2 (k) \otimes M \ar@{^{(}->}[r] & F_1 \otimes M \ar[r]^{\partial_1 \otimes \id} & F_0 \otimes M \ar[r]^{\epsilon \otimes \id} & M \ar@{=}[d] \\
M \ar[u]^\mu \ar[r]^{\iota} & X_1 \ar[u]^{\theta_1} \ar[r]^d & X_0 \ar[u]^{\theta_0} \ar[r]^\pi & M . }
\]

Here $\mu$ is a split monomorphism.  The condition on $f$ means that $f = (g \otimes \id) \circ \mu$ for some $g \in \Hom_\Lambda (\Omega^2(k), k)$, and we assume $g \neq 0$.  We then get a sequence representing $f$ as follows:

\[\xymatrix{
0 \ar[r] & \Omega^2 (k) \otimes M \ar[r] \ar[d]^{g \otimes \id} & F_1 \otimes M \ar[r] \ar[d]^{q \otimes \id} & F_0 \otimes M \ar[r] \ar@{=}[d]& M \ar[r] \ar@{=} [d]& 0 \\
0 \ar[r] & M \ar[r] ^\alpha & \frac{F_1}{\ker g} \otimes M \ar[r]^{\bar{\partial_1} \otimes \id} & F_0 \otimes M \ar[r] & M \ar[r] & 0
}\]
where $q$ is the quotient map and $\alpha(m) = q(y) \otimes m$ where $y$ is some fixed element of $\Omega^2(k)$ such that $g(y)=1$.  Now $\xi^2$ is represented by $\Omega^2(f) \circ f$, and also by the Yoneda product of two copies of the bottom row of the diagram above.  As a map $M \rightarrow M$ the Yoneda composition corresponds to $f^2$ because the following diagram commutes:

\[\xymatrix{
 M \ar@{^{(}->}[r] \ar[d]^{f^2} & X_1 \ar[r] \ar[d]^{(q\otimes f) \circ \theta_1} & X_0 \ar[r]^{\iota \circ \pi} \ar[d]^{(\id \otimes f) \circ \theta_0} & X_1 \ar[r] \ar[d]^{q \circ \theta_1}& X_0 \ar@{->>}[r] \ar[d]^{\theta_0} & M  \\
 M \ar@{^{(}->}[r]^\alpha & \frac{F_1}{\ker g} \otimes M \ar[r]^{\bar{\partial_1} \otimes \id} & F_0\otimes M \ar[r]^{\alpha \circ (\epsilon \otimes \id)} & \frac{F_1}{\ker g} \otimes M \ar[r]^{\bar{\partial_1} \otimes \id} & F_0 \otimes M \ar@{->>}[r]^{\epsilon \otimes \id} & M  \ar@{=}[u]
}\]
It follows that $\Omega^2 ( f ) \circ f = f^2$ in the stable category.  Replacing $f$ by $\id+f$ if necessary we may assume that $f$ is invertible (recall that $M$ is indecomposable so its endomorphism ring is local).  The result follows.
\end{proof}

Our second lemma is a reciprocity result for Hopf algebras.

\begin{lem} \label{reciprocity}
Let $\Gamma$ be a subalgebra of $\Lambda$, so $\Gamma$ has a ``trivial module'' $k_\Gamma$ defined by the counit $\epsilon$ of $\Lambda$ restricted to $\Gamma$.  Then for any $\Lambda$-module $M$ we have
\[ (M|_\Gamma)\!\! \uparrow ^\Lambda \cong (k_\Gamma \!\! \uparrow ^\Lambda )\otimes M. \]
\end{lem}
\begin{proof}
The map is given by
\[ \lambda \otimes _ \Gamma m \mapsto \sum (\lambda _{(1)} \otimes _\Gamma 1) \otimes \lambda_{(2)} m. \]
It is well-defined because for $\lambda \in \Lambda$ and $\gamma \in \Gamma$ we have
\begin{align*} \lambda\gamma \otimes _\Gamma m \mapsto  & \sum (\lambda_{(1)} \gamma_{(1)} \otimes _\Gamma 1) \otimes \lambda_{(2)} \gamma_{(2)} m \\
= & \sum (\lambda_{(1)} \otimes \epsilon(\gamma_{(1)})) \otimes \lambda_{(2)} \gamma_{(2)}m \\
= & \sum (\lambda_{(1)} \otimes _\Gamma 1) \otimes \lambda_{(2)} \epsilon(\gamma_{(1)}) \gamma_{(2)} m \\
= & \sum (\lambda_{(1)} \otimes _\Gamma 1) \otimes \lambda_{(2)}\gamma m .\end{align*}
which is the image of $\lambda \otimes _\Gamma \gamma m$.  It is easy to check that this is a module homomorphism with inverse
\[ (\lambda \otimes_\Gamma 1) \otimes m \mapsto \sum \lambda_{(1)} \otimes S(\lambda_{(2)}) m. \]
\end{proof}

The same result holds with $M \otimes (k_\Gamma \!\! \uparrow ^\Lambda)$ instead of $(k_\Gamma \!\! \uparrow ^\Lambda) \otimes M$, showing that these two modules are isomorphic even if $\Lambda$ is not cocommutative.

\begin{rem}
 When $\Gamma$ is a sub-Hopf algebra of $\Lambda$ there is a stronger result, that $(N \otimes M|_\Gamma)\!\! \uparrow ^\Lambda \cong (N\!\! \uparrow ^\Lambda) \otimes M$ for any $\Gamma$-module $N$.  In this case the map is given by
\[\lambda \otimes_\Gamma (n \otimes m) \mapsto \sum  (\lambda_{(1)} \otimes_\Gamma n) \otimes \lambda _{(2)} m. \]
\end{rem}

\end{subsection}

\section{Detecting projectivity} \label{detecting_proj_section}
In this section we introduce a condition on the cohomology of a
Hopf algebra $\Lambda$ that is sufficient to determine a set of
subalgebras of $\Lambda$ with the property that a $\Lambda
$-module is projective if and only if it is projective when restricted
to any subalgebra of this set. This condition will be very
useful for defining rank varieties for certain classes of Hopf
algebras.
\begin{defn}
An element $\xi$ of $\ext_\Lambda ^n (M,N)$ is of \textbf{induced type} if it can be represented by a sequence of the form
\[
 0 \rightarrow M \rightarrow L_{n-1} \otimes (k_{H_{n-1}} \! \!\uparrow ^\Lambda) \rightarrow \cdots \rightarrow L_0 \otimes (k_{H_0}\!\!\uparrow ^\Lambda ) \rightarrow N \rightarrow 0
\]
for some modules $L_0, \ldots, L_{n-1}$ and some proper subalgebras $H_0, \ldots, H_{n-1}$ of $\Lambda$.  The $H_i$ are said to be \textbf{involved} in $\xi$.

\end{defn}

\begin{defn}
 A subalgebra $\Gamma$ of $\Lambda$ is called a \textbf{flat subalgebra} if $\Lambda$ is projective as a left $\Gamma$-module.
\end{defn}

\begin{defn} A collection of subalgebras $\mathcal{A}$ of $\Lambda$ is said to \textbf{detect projectivity} if a $\Lambda$-module $M$ is projective if and only if it is projective on restriction to every $A \in \mathcal{A}$.
\end{defn}

If $\mathcal{A}$ detects projectivity every element of $\mathcal{A}$ must be a flat subalgebra of $\Lambda$.

\begin{theo} \label{detecting_thm}
Suppose there exists a two-dimensional subspace of $\ext ^n _\Lambda (k,k)$, where $n \geq 2$, such that every element of this subspace is of induced type and the subalgebras involved can be taken to be flat.  Let $\mathcal{A}$ be the collection of these flat subalgebras.  Then $\mathcal{A}$ detects projectivity.
\end{theo}
The proof is based on Carlson's ideas from \cite{carl:vars}.
\begin{proof}
If $M$ is projective then it is projective on restriction to any $A \in \mathcal{A}$ by the flatness condition.  So suppose $M$ is indecomposable and projective on restriction to every element of $\mathcal{A}$; we wish to show that $M$ is a projective $\Lambda$-module.

Let $\xi_1$ and $\xi_2$ span the subspace given by the hypothesis of the theorem.  These are of induced type, and by tensoring with $M$ and using Lemma \ref{reciprocity} we get a sequence
\[
 C_{\xi_1} \otimes M : 0 \rightarrow M \rightarrow L_{n-1} \otimes  (M|_{H_{\xi_{n-1}}})\!\! \uparrow ^\Lambda \rightarrow \ldots \rightarrow
L_0 \otimes (M|_{H_{\xi_1}'} )\!\!\uparrow^\Lambda \rightarrow M \rightarrow 0
\]
where $H_{\xi_1}$ and $H_{\xi_1} '$ are in $\mathcal{A}$.  The middle terms are projective because induction preserves projectivity, so $\Omega^n(M) \cong M$ and for all $l \in \Z$ we have that $\ext^{nl}_\Lambda (M,M)$ is isomorphic to $\stend(M)$.

This means that $C_{\xi_1} \otimes M$ and $C_{\xi_2} \otimes M$ correspond to elements $f_1$ and $f_2$ of $\stend(M)$.  This ring is local because $M$ is indecomposable, and so some linear combination $f = \alpha f_1 + \beta f_2$ is nilpotent. The map $f$ represents $C_{\alpha \xi_1 + \beta \xi_2} \otimes M$.  By the remark after Lemma \ref{omegaf}, powers of this correspond to powers of $f$, which are eventually zero.  But repeated Yoneda splices of $C_{\alpha \xi_1 + \beta \xi_2} \otimes M$ begin a projective resolution of $M$, so they can never be zero unless $M$ is projective, as $\Lambda $ is self-injective.
\end{proof}

Let $\mathcal{A}$ be a set of subalgebras of $\Lambda$. We define the map $V: \Lambda$-mod$  \to \mathcal{A}  \cup \{0\}$ by 
\[ V(M)=\{ 0 \} \cup \{ A \in \mathcal{A}:  M|_{A} \mbox{ is not projective }  \} \]
\begin{lem}\label{detect}
Let $\mathcal{A}$ be a set of self-injective subalgebras of $\Lambda$ that detect projectivity. Then $V$ as defined above satisfies (C1)-(C4).
\end{lem}

\begin{proof}
Property (C1) follows immediately from the definition, and (C4) holds by hypothesis. As any $A \in \mathcal{A}$ is self-injective, we have that the
exact sequence \[0 \to M_1|_A \to M_2|_A \to M_3|_A \to 0\] splits
if any two modules $M_s|_A$, $M_j|_A$ are projective. Then
$M_i|_A$ is also projective and therefore (C2) holds. Part (C3)
is proved by induction. For $n=1$ we consider the exact sequence $0
\to \Omega(M) \to P \to M \to 0$, where $P$ is a projective cover
of $M$. Then $P|_A$ is projective and (C2) proves the result.
\end{proof}

\section{The Drinfel'd double of the Taft algebra when $n=2$}

We apply the results of the previous section to the case of the
Drinfel'd double $D(\Lambda_2)$ of the Taft algebra when $n=2$,
over an algebraically closed field $k$ of characteristic not equal
to $2$. This algebra is given by generators and relations as follows:
\begin{eqnarray*} D(\Lambda_2) & = & \langle x,X,g,G | x^2,\  X^2,\  g^2=1,\  G^2=1,\  gG=Gg,\  gx=-xg, \\ & & gX=-Xg,\  Gx=-xG,\  GX=-XG,\  xX+Xx= 1-gG \rangle. \end{eqnarray*}
$D(\Lambda_2)$ has a Hopf structure with respect to which the elements $g$ and $G$ are grouplike, and the coproduct of $x$ and $X$ is given by
\begin{eqnarray*}
\Delta(x) &=& 1 \otimes x + x \otimes g \\
\Delta(X) &=& 1 \otimes X + X \otimes G
\end{eqnarray*}
The counit sends $g$ and $G$ to $1$ and kills $x$ and $X$.  

\begin{lem} \label{semisimple_block}
The elements  $f_+=(1/2)(1+gG) $ and $f_-=(1/2)(1-gG)$ are central primitive idempotents in $D(\Lambda_2)$, and the block $D(\Lambda_2)f_-$ is semisimple. 
\end{lem}

\begin{proof}
 Only the semisimplicity needs checking.  If $E = xf_-$ and $F=Xf_-$ then $EF + FE = 2f_-$, so that the subalgebra generated by $E$ and $F$ is isomorphic to $M_2(k)$ via 
\[ E \mapsto \left(\begin{array}{ll} 0&0\\1&0 \end{array}\right) \,\,\, F \mapsto \left(\begin{array}{ll} 0&2\\0&0 \end{array}\right).\] 
Since $D(\Lambda_2)f_- = \langle E, F \rangle \rtimes \langle gf_- \rangle $ and $k\langle g \rangle$ is semisimple, we have that $D(\Lambda_2)f_- \cong M_2(k)\oplus M_2(k)$.
\end{proof}

For $\alpha, \beta \in k$ not both zero we define $H_{\alpha \beta}$ to be $\langle g,G,\alpha x + \beta X \rangle$.  The reason for introducing these subalgebras is that they verify the conditions of Theorem \ref{detecting_thm} for $D(\Lambda_2)$, which we will now check.

\begin{lem} $D(\Lambda_2)$ is free as a left $H_{\alpha \beta}$-module. \end{lem}
\begin{proof}
Put $H = H_{\alpha \beta}$.  It is easy to verify that 
\[ D(\Lambda_2)|_{H} = H \oplus H\cdot (\beta x - \alpha X) \]
and the second factor is isomorphic to $H$ as a left $H$-module.
\end{proof}
Let $k$ be the trivial $D(\Lambda_2)$-module, and $k_-$ the the $1$-dimensional $D(\Lambda_2)$-module on which $g$ and $G$ act as $-1$.  

\begin{lem}  \label{induced_mods} $k_{H_{\alpha \beta}}\!\! \uparrow ^{D(\Lambda_2)}$ is an indecomposable $2$-dimensional module with top $k$ and socle $k_-$, and any such $2$-dimensional module is induced from the trivial module of some $H_{\alpha \beta}$.  For any $\alpha, \beta, \gamma, \delta$ with $(\alpha, \beta) \neq (0,0) \neq (\gamma, \delta)$ there is a non-split exact sequence
\[0 \rightarrow k \rightarrow k_{H_{\alpha \beta}}\!\! \uparrow ^{D(\Lambda_2)} \otimes k_- \rightarrow k_{H_{\gamma \delta}}\!\! \uparrow ^{D(\Lambda_2)} \rightarrow k \rightarrow 0. \]
\end{lem}

\begin{proof}
Let $H = H_{\alpha \beta}$ and let $\alpha \neq 0$.  It is easy to verify that $\alpha \otimes 1$ and $ X \otimes 1$ form a basis for $k_H\!\! \uparrow ^{D(\Lambda_2)} = D(\Lambda_2) \otimes _H k_H$, and that $X \otimes 1$ generates a submodule isomorphic to $k_-$.

Given any $2$-dimensional indecomposable module with top $k$ and socle $k_-$, we may choose a basis such that the elements $x$ and $X$ must act via matrices of the form
\[ \left(\begin{array}{ll} 0 & -\beta \\ 0 & 0\end{array} \right) \,\,\, \mathrm{and} \,\,\,\left(\begin{array}{ll} 0 & \alpha \\ 0 & 0\end{array} \right). \]
If $\alpha \neq 0$ the action of $D(\Lambda_2)$ is the same as on the basis for $k_{H_{\alpha \beta}} \!\! \uparrow ^{D(\Lambda_2)}$ given above, so these modules are isomorphic. Similar arguments hold for $\beta \neq 0$.
The claim about exact sequences follows immediately.
\end{proof}

To apply Theorem \ref{detecting_thm} to the family $\mathcal{A} = \{ H_{\alpha \beta} : (\alpha, \beta) \neq (0,0) \}$ we need to verify the condition on $\ext ^2 _{D(\Lambda_2)}(k,k)$.  We will prove a stronger result that allows us to deduce some finite generation properties for certain modules over $\ext ^* _{D(\Lambda_2)}(k,k)$.

\begin{lem} \label{induced_type_lem} Every element of $\ext_{D(\Lambda_2)}^n(k,k)$, $\ext_{D(\Lambda_2)} ^n (k,k_-)$, $\ext_{D(\Lambda_2)}^n(k_-,k)$ and $\ext_{D(\Lambda_2)} ^n (k_-, k_-)$ is of induced type with all algebras involved of the form $H_{\alpha \beta}$. \end{lem}

\begin{proof}
Let $e_\pm =(1/4)( 1 \pm g \pm G + gG)$, so that $D(\Lambda_2)e_+$ and $D(\Lambda_2)e_-$  are the projective covers $P_+$, $P_-$ of $k$ and $k_-$. These modules are spanned by $e_\pm, xe_\pm, Xe_\pm, xXe_\pm$ and have the Loewy structures
\[ \begin{tabular}{ll}
    \multicolumn{2}{c}{$k$} \\
$k_-$ & $k_-$ \\
\multicolumn{2}{c}{$k$}
   \end{tabular}
\,\,\,\mathrm{and}\,\,\,
\begin{tabular}{ll}
    \multicolumn{2}{c}{$k_-$} \\
$k$ & $k$ \\
\multicolumn{2}{c}{$k_-$}
   \end{tabular}\]
respectively.  It follows that there is a minimal projective resolution of $k$ given by
\begin{equation} \label{mpr} \ldots \rightarrow P_+^{\oplus 2i+1} \stackrel{\partial_{2i}}{\rightarrow} P_{-}^{ \oplus 2i} \stackrel{\partial_{2i-1}}{\rightarrow} P_+^{\oplus 2i-1} \rightarrow \ldots \rightarrow  P_+ \stackrel{\partial_0}{\rightarrow} k \rightarrow 0, \end{equation}
and a minimal projective resolution for $k_-$ can be obtained by
tensoring this sequence with $k_-$, that is, swapping $P_+$ and
$P_-$.  We see that the groups $\ext_{D(\Lambda_2)} ^n (k_-,k)$
and $\ext_{D(\Lambda_2)} ^n (k,k_-)$ are non-zero if and only if
$n$ is odd, and the groups $\ext_{D(\Lambda_2)} ^n (k,k)$ and
$\ext_{D(\Lambda_2)} ^n (k_-,k_-)$ are non-zero if and only if $n$
is even.  A non-zero element of one of these groups is represented
by a non-zero homomorphism $\phi$ from the $n$th kernel $\Omega^n$
of one of these resolutions to $k$ or $k_-$.  To find a sequence
representing this element we for the push out diagram:
\[ \xymatrix{
\Omega ^n \ar@^{(->}[r] \ar[d]^\phi & P_\mp ^{\oplus n} \ar[d]^q \\
k_\pm  \ar[r]                          & P_\mp^{\oplus n} / \ker \phi
}
\]
where the horizontal map on the bottom row sends $1$ to the image of an element $w$ of $\Omega ^n$ such that $\phi(w)=1$, and $q$ is the quotient map.  Let $w = (w_1, \ldots, w_n)$ where $w_i$ lies in the $i$th summand of $P_\mp ^{\oplus n}$ and without loss of generality $w_1 \notin \ker \phi$.  Let $w'$ span a complement to the image of $w_1$ in the radical of the first summand of $P_\mp ^{\oplus n}$ modulo its socle. Then if $I$ is the submodule of $P_\mp ^{\oplus n}$ generated by $w'$, $\ker \phi$ and $(0,P_\mp, \ldots, P_\mp)$ we have that $P_\mp ^{\oplus n}/I$ is $2$-dimensional with socle $k_\pm$ and top $k_\mp$.  By Lemma \ref{induced_mods} this is a module of the form $k_{H_{\alpha \beta}} \!\! \uparrow ^{D(\Lambda_2)}$, and hence we obtain a commutative square
\[ \xymatrix{
\Omega ^n \ar@^{(->}[r] \ar[d]^\phi & P_\mp ^{\oplus n} \ar[d] \\
k_\pm  \ar[r]                          & P_\mp^{\oplus n} / I \cong k_{H_{\alpha \beta}} \!\! \uparrow ^{D(\Lambda_2)} \otimes k_\mp
}
\]
where the horizontal map on the bottom row is non-zero.

To finish the proof the lemma it is enough to show that any diagram

\[ \xymatrix{
P_\mp ^{\oplus n} \ar[r]^{\partial_{n-1}} \ar[d]^\gamma & P_\pm ^{\oplus n-1}  \\
k_{H_{\alpha \beta}} \!\! \uparrow ^{D(\Lambda_2)} \otimes k_\mp  &
}
\]
with first vertical map non-zero can be completed to a commutative square with the bottom right module of the form $k_{H_{\alpha'  \beta'}} \!\! \uparrow ^{D(\Lambda_2)} \otimes k_\pm$ and the second vertical map non-zero.  This follows by a pushout argument very similar to the above.
\end{proof}

\begin{cor}\label{finitely generated}
The ring $\ext_{D(\Lambda_2)} ^* (k,k_-)$ is generated as a left $\ext_{D(\Lambda_2)} ^* (k_-,k_-)$-module and as a right $\ext_{D(\Lambda_2)}^*(k,k)$-module by $\ext_{D(\Lambda_2)} ^1 (k,k_-)$.  A similar result holds for $\ext_{D(\Lambda_2)} ^* (k_-,k)$.
\end{cor}
\begin{proof}
 The module action is via Yoneda product.  Any non-zero element of $\ext_{D(\Lambda_2)} ^n (k,k_-)$ is represented by a sequence of the form
\[ 0 \rightarrow k_-  \rightarrow k_{H_1}\!\! \uparrow ^{D(\Lambda_2)}  \rightarrow k_{H_2}\!\! \uparrow ^{D(\Lambda_2)} \otimes k_- \rightarrow \ldots \rightarrow  k_{H_n}\!\! \uparrow ^{D(\Lambda_2)} \rightarrow k \rightarrow 0 \]
where $H_i = H_{\alpha _i \beta _i}$ for some $\alpha_i, \beta_i \in k$, which is visibly the Yoneda splice of an element of $\ext^{n-1}_{D(\Lambda_2)} (k_-, k_-)$ followed by an element of $\ext_{D(\Lambda_2)} ^1 (k,k_-)$.  The other parts of the lemma follow similarly.
\end{proof}

We are now in a position to define rank varieties for $D(\Lambda_2)$.

\begin{defn}
 Let $M$ be a $D(\Lambda_2)$-module.  Then the \textbf{rank variety} $V^r(M)$ is defined to be
\[ \{ 0 \} \cup \{ (\alpha, \beta) \in k^2 : M |_{H_{\alpha \beta}} \mathrm{is\,\, not\,\, projective} \}. \]
\end{defn}
Note that this is a union of lines through the origin in $k^2$, so there is a corresponding projective variety $\bar{V}^r(M)$.  That this variety satisfies Dade's Lemma is immediate from Theorem \ref{detecting_thm}.
In fact, by Lemma \ref{detect} the rank variety $V^r$ satisfies $(C1)-(C4)$. 

Next we show that the rank variety is well-behaved under tensor products.

\begin{lem} \label{tensor}
Let $M$ and $N$ be finite-dimensional $D(\Lambda_2)$-modules. Then 
\[V^r(M\otimes N)=V^r(M) \cap V^r(N).\]
\end{lem}

\begin{proof} The result is equivalent to showing that $(M\otimes N)|_{H_{\alpha \beta}}$ is projective if and only if $M|_{H_{\alpha \beta}}$ or $N|_{H_{\alpha \beta}}$ is projective.  Without loss of generality we choose $M$ and $N$ to be indecomposable. If $f_-M \neq 0$ then $M$ is simple projective by Lemma \ref{semisimple_block}, so $M\otimes N$ is projective and the lemma follows. So we may assume $M= f_+M$ and $N=f_+N$.  

It is therefore enough to prove that $(M\otimes N)|_{H_{\alpha \beta}f_+}$ is projective if and only if $M|_{H_{\alpha \beta}f_+}$ or $N|_{H_{\alpha \beta}f_+}$ is projective.  The subalgebra $H:=H_{\alpha\beta}f_+$ is isomorphic to $k\langle s \rangle \ltimes k[t]/ t^2 $, where $t$ corresponds to $t=(\alpha x+ \beta X)f_+$ and $s$ to $ gf_+=Gf_+$, so $sts=-t$.

The coproduct on $D(\Lambda_2)$ acts as follows:
\begin{align} \label{coprod_t}
\Delta( t) & =(\alpha x  \otimes g +\beta X \otimes G + 1\otimes (\alpha x+ \beta X))(f_+\otimes f_+ + f_-\otimes f_-)\\ \nonumber
\Delta( s) & =(g \otimes g )(f_+\otimes f_+ + f_-\otimes f_-).
\end{align}
Therefore $t$ acts on $M\otimes N$ as $t\otimes s +1\otimes t$, $s$ acts as $s\otimes s$ and $f_+$ as the identity. Then $H$ is a Hopf algebra with comultiplication $\Delta(t)=t\otimes s+1\otimes t$, $\Delta(s)=s \otimes s$, counit $\epsilon(s)=1$, $\epsilon(t)=0$ and antipode $S(s)=s,$ $S(t)=-ts$. With this action we have $(M\otimes N)_H \cong M|_H \otimes N|_H$. Therefore $(M\otimes N)_H$ is projective if and only if $M|_H$ is projective or $N|_H$ is projective. The statement then follows immediately by the definition of the rank variety. 
\end{proof}

\begin{subsection}{The restriction map $\res^{ D(\Lambda_2)}_{ H_{\alpha \beta}}$} \label{restriction}
We will need later to understand the map $\res^{ D(\Lambda_2)}_{ H_{\alpha \beta}} : \ext_{D(\Lambda_2)} ^*(k,k) \rightarrow \ext ^* _{H_{\alpha \beta}} (k,k)$ induced by the inclusion $H_{\alpha \beta} \hookrightarrow D(\Lambda_2)$.  Let $\mathsf{P}$ be the minimal $D(\Lambda_2)$-projective resolution (\ref{mpr}) of $k$
and let $\mathsf{Q}$ be a minimal $H_{\alpha \beta}$-projective
resolution of $k$. In order to compute the restriction map we compute a chain map lifting the identity $k \rightarrow k$ to a map $\mathsf{Q} \rightarrow \mathsf{P}|_{H_{\alpha \beta}}$. From the proof of Lemma \ref{induced_type_lem} we see that
\begin{equation} \label{Qpm}
P_+ |_{H_{\alpha \beta}} = \langle e_+, (\alpha x + \beta X)e_+ \rangle \oplus \langle (\alpha X - \beta x) e_+, xXe_+ \rangle 
\end{equation}
where the first summand is isomorphic to the $H_{\alpha \beta}$-projective cover $Q_+$ of
$k$ and the second to the $H_{\alpha \beta}$-projective cover
$Q_-$ of $k_-$.  A similar result holds for the restriction of
$P_-$.  It follows that a minimal $H_{\alpha \beta}$-projective
resolution of $k$ can be obtained by taking repeated Yoneda
splices of
\[ 0 \rightarrow k \rightarrow Q_- \rightarrow Q_+ \rightarrow k \rightarrow 0 \]
and so $\ext ^* _{H_{\alpha \beta}} (k,k) \cong k[\gamma] $ where $\gamma$ is the element of $\ext ^2 _{H_{\alpha \beta}} (k,k)$ represented by the sequence above.

We now give the lift between $\mathsf{Q} $ and $\mathsf{P}|_{H_{\alpha \beta}}$:

\[ \xymatrix{
P_+ \oplus P_+ \oplus P_+ \ar[r] & P_- \oplus P_- \ar[r] & P_+ \ar[r] & k \ar[r] & 0 \\
Q_+ \ar[r] \ar[u] & Q_- \ar[r] \ar[u] & Q_+ \ar[r] \ar[u] & k \ar[r] \ar@{=}[u] & 0
 }
\]
The horizontal maps on the top row are determined by 
\[\begin{array}{ll}
(ae_+, be_+ ce_+)  \mapsto ((ax+ bX)e_-, (bx + c X)e_-) & \\ 
(ae_-, be_-)       \mapsto (ax + bX)e_+ &
e_+                \mapsto 1
\end{array}\]
reading from left to right. Identifying $Q_+$ with the first summand in (\ref{Qpm}) and $Q_-$
with the corresponding summand of $P_- | _{H_{\alpha \beta}}$ the
horizontal maps on the bottom row are $e_+ \mapsto (\alpha x +
\beta X) e_-$, $e_- \mapsto (\alpha x + \beta X) e_+$ and $e_+
\mapsto 1$.  The vertical maps are
$e_+ \mapsto (\alpha^2 e_+,  \alpha \beta e_+, \beta ^2 e_+)$, $e_- \mapsto (\alpha e_-, \beta e_-)$ and $e_+ \mapsto e_+$.

The ring structure of $\ext^* _{D(\Lambda_2)} (k,k)$ was computed by Taillefer in \cite{taillefer}: 
\[ \ext ^*_{D(\Lambda_2)} (k,k) = k[x,y,z] / (xy-z^2) \]   where $x,y,z$ all have degree 2.  These correspond to homomorphisms $u_x,u_y,u_z$ from $P_+ \oplus P_+ \oplus P_+$ to $k$:
\[\begin{array}{lll}
u_x  (a,b,c) = a\cdot 1 &
u_y (a,b,c) = c \cdot 1 &
u_z(a,b,c) = b \cdot 1 
\end{array}\]
Our computations show that
\[\begin{array}{lll}
\res^{ D(\Lambda_2)}_{ H_{\alpha \beta}} x = \alpha^2 \gamma & 
\res^{ D(\Lambda_2)}_{ H_{\alpha \beta}} y = \beta ^2 \gamma &
\res^{ D(\Lambda_2)}_{ H_{\alpha \beta}} z = \alpha \beta \gamma. 
\end{array}\]

The map $k^2 \rightarrow \operatorname{\mathrm{maxspec}} \ext^* _{D(\Lambda_2)} (k,k) = \{ (a,b,c) \in k^3 : ab=c^2 \}$ induced by restriction is therefore a bijection, but not invertible as a map of varieties (the inverse will involve a square root).  This is the same as the situation for elementary abelian groups: see \cite[Remark after 5.8.2]{bensonII}.

\begin{rem}\label{ext}
Given a $k$-algebra $A$, a group $G$ such that $kG$ is semisimple, and a map $G \rightarrow \aut A$, form the split extension $B = A \rtimes kG$.  For any $B$-module $M$ there is a ring isomorphism
\[ \ext ^* _B(M,M) \cong \ext^*_A (M|_A, M|_A) ^G \]
where the right hand side is the ring of $G$-invariants.  This can be used to find the ring structure of $\ext^* _{D(\Lambda_2)} (k,k)$ as follows: as shown in Section \ref{n>2section}, the principal block of $D(\Lambda_2)$ is Morita equivalent to the split extension of $A = k[x_1, x_2]/(x_1^2, x_2^2)$ by a cyclic group $\langle g \rangle$ of order 2 that conjugates each generator to its additive inverse.  $\ext^*_A(k,k)$ is isomorphic to $k[\eta_1, \eta_2]$ with the $\eta_i$s in degree $1$, and the action of $g$ multiplies $\eta_i$ by $-1$.  Therefore the ring of invariants is generated by $\eta_1 ^2$, $\eta_2^2$ and $\eta_1 \eta_2$ and is isomorphic to $k[x,y,z]/(xy-z^2)$.
\end{rem}

\end{subsection}

\begin{subsection}{Connecting the rank and support varieties}
An important property of the rank variety is that it can be identified with the support variety.  We will link the definition of rank variety to the existing definition of
the support variety on $D(\Lambda_2)$.

Note that by the results of the previous section, $\ext^*_{H_{\alpha \beta}}(k,k) \cong k[\gamma]$ where $\gamma$ is in degree $2$. We write $\tau_{\alpha \beta}$ for the restriction map $\res^{ D(\Lambda_2)}_{ H_{\alpha \beta}}$ which is surjective by \ref{restriction}. The kernel of $\tau_{\alpha \beta}$ is therefore in $\operatorname{Proj}\ext^{ev}_{D(\Lambda_2)}(k,k)$, as any subring of $k[\gamma]$ does not have zero divisors, and it depends only on the line $\overline{(\alpha,\beta)}$ through the origin and $(\alpha, \beta)$.

\begin{lem}
Let $\zeta \in \ext^*_{D(\Lambda_2)}(k,k)$ and $M$ be an indecomposable $D(\Lambda_2) $-module. Then $(\zeta \otimes M)|_{H_{\alpha \beta}}=0$ if and only if $M|_{H_{\alpha\beta}}$ is projective or $\tau_{\alpha \beta} (\zeta)=0$. 
\end{lem}

\begin{proof}
As $M$ is indecomposable we can assume without loss of generality that $f_+M =M$. Recall $H_{\alpha \beta} f_+ \cong k\langle s \rangle \ltimes k[t ]/t^2$, so it has two simple one-dimensional modules which we denote by $k$ and $k_-$.
Therefore $M|_{H_{\alpha \beta}f_+}$ is the direct sum of a projective module $M_p$ and a semisimple module $M_n$.
Using (\ref{coprod_t}) we have \[(\zeta \otimes M)|_{H_{\alpha\beta}f_+}= (\zeta\otimes M_n)|_{H_{\alpha\beta}f_+} \oplus (\zeta \otimes M_p)|_{H_{\alpha\beta}f_+}\] as exact sequences.

Clearly $(k\otimes M_p)|_{H_{\alpha\beta}f_+}\cong M_p$ which is projective (hence injective) as a $H_{\alpha \beta}f_+$-module. So the
sequence $(\zeta \otimes M_p)|_{H_{\alpha\beta}f_+}$ splits. As
$M_n$ decomposes as a direct sum of one-dimensional simple
modules, $(\zeta\otimes M_n)|_{H_{\alpha\beta}f_+}$ is a direct product of copies of $\tau_{\alpha \beta }(\zeta)$ and $\tau_{\alpha \beta }(\zeta) \otimes k_-$. So $(\zeta \otimes M)|_{H_{\alpha\beta}f_+}=0$ if and only if $M|_{H_{\alpha \beta}f_+}$ is projective or $\tau_{\alpha \beta }(\zeta)=0$.
\end{proof}

\begin{theo}
For any $D(\Lambda_2)$-module $M$, the map $\phi_M: \bar{V}^r(M) \to V^s(M)$ sending $\overline{(\alpha, \beta )}$ to $\ker
\tau_{\alpha \beta}$ is bijective.
\end{theo}
\begin{proof}
Let $(\alpha, \beta)\in V^r(M)$, then $M|_{H_{\alpha\beta}}$ is not projective. Let $\zeta \in \ker \psi_M$ be a homogeneous element. Then $(\zeta\otimes M)|_{H_{\alpha\beta}}=0$ and by the
previous lemma we have $\tau_{\alpha \beta}(\zeta)=0$. This shows
that $\ker \psi_M \subseteq \ker \tau_{\alpha\beta}$, so the map
$\phi_M$ sends $\bar{V}^r(M)$ to $V^s(M)$ and is well defined.

The calculation of $\res^{D(\Lambda_2)}_{H_{\alpha\beta}}$ given earlier shows that the map is injective, so it remains to show surjectivity. Again using the calculation we have that every element of $V^s(k)$ is of the form $\ker \tau_{\alpha \beta }$. Let $\Ann(M,M) \subseteq \ker \tau_{\alpha \beta}$. We need to show that $(\alpha, \beta ) \in V^r(M)$. There are homogeneous elements $\zeta_1, \cdots, \zeta_n$ such that $\ker \tau_{\alpha \beta}=\cap_{i=1} ^n \langle \zeta_i \rangle$. We set $L_{\zeta} = L_{\zeta_1} \otimes \cdots \otimes L_{\zeta_n}$, then $V^s(L_{\zeta} \otimes M)=  \ker \tau_{\alpha \beta} = \phi_M \overline{(\alpha, \beta)} $ by applying Lemma \ref{zeta} repeatedly. The first part of the proof then implies that $\overline{(\alpha, \beta)} \in \bar{V}^r( L_{\zeta} \otimes M )$. Therefore $(\alpha, \beta) \in V^r(M)$ by Lemma \ref{tensor}.
\end{proof}

\end{subsection}

\section{The Drinfel'd double of the Taft algebra when $n>2$ and $\uqsl$} \label{n>2section}

Let $k$ be a field whose characteristic does not divide $n$, which
contains a primitive $n$th root of unity $q$.  The Drinfel'd double
of the Taft algebra over $k$ is presented as follows:
\begin{eqnarray*} D(\Lambda_n) & = & \langle x,X,g,G | x^n,\ X^n,\ g^n=1,\ G^n=1,\ gG=Gg,\ gx=q^{-1}xg, \\ & & gX=qXg,\ Gx=q^{-1}xG,\ GX=qXG,\ xX-qXx= 1-gG \rangle \end{eqnarray*}
This algebra has a Hopf structure with respect to which the elements $g$ and $G$ are grouplike, and the coproduct of $x$ and $X$ is given by
\begin{eqnarray*}
 \Delta(x) &=& 1 \otimes x + x \otimes g \\
\Delta(X) &=& 1 \otimes X + X \otimes G.
\end{eqnarray*}
The antipode $S$ inverts $g$ and $G$, and satisfies $S(x)=-xg^{-1}$ and $S(X) = -XG^{-1}$
The counit $\epsilon$ has $\epsilon(g)=\epsilon(G)=1$ and $\epsilon(x)=\epsilon(X)=0$.

$\uqsl$ is a finite-dimensional quotient of the quantised enveloping algebra of $\mathfrak{sl_2}$.  There is some variation in notation and defining relations for this algebra in the literature (\cite[VI.5]{kassel}, \cite{suter}, \cite{jie}, \cite{patra}).

According to \cite{EGST} and \cite{suter}, a non-semisimple block $\mathcal{B}$ of $D(\Lambda_n)$ or of $\uqsl$ contains exactly two isomorphism classes of simple modules, $S$ and $T$ say.  The Loewy structures of the projective covers $P_S$ and $P_T$ are
\[ \begin{tabular}{ll}
    \multicolumn{2}{c}{$S$} \\
$T$ & $T$ \\
\multicolumn{2}{c}{$S$}
   \end{tabular}
\,\,\,\mathrm{and}\,\,\,
\begin{tabular}{ll}
    \multicolumn{2}{c}{$T$} \\
$S$ & $S$ \\
\multicolumn{2}{c}{$T$}
   \end{tabular} \]
respectively.  The projective module $P=P_S \oplus P_T$ is a progenerator of $\mathcal{B}$, which is therefore Morita equivalent to $\End_\mathcal{B}(P) ^{\mathrm{op}}$.

\begin{lem}
$\End_\mathcal{B}(P) ^{\mathrm{op}}$ is isomorphic to the algebra $A$ given by
\begin{equation}\label{A_presentation} k \langle y_1,y_2 \rangle /\langle y_1^2, y_2^2, y_1y_2+y_2y_1  \rangle  \rtimes k\langle g \rangle 
\end{equation}
where $g^2=1$ and $gy_ig^{-1}=-y_i$.
\end{lem}
The algebra $A$ is a split extension of the quantum symmetric algebra \cite{BEH} generated by $y_1$ and $y_2$ by the group algebra of a cyclic group $\langle g \rangle$ of order $2$.

\begin{proof}
By replacing $y_2$ with $x_2:=gy_2$ and setting $x_1:=y_1$ we see that $A$ is isomorphic to 
\[ A' =  k[x_1,x_2]/\langle x_1^2, x_2^2 \rangle  \rtimes k\langle g \rangle\] where $g^2=1$ and $gx_ig^{-1}=-x_i$, 
a split extension of the Kronecker algebra by a cyclic group of order 2.
Suter \cite[\S 5]{suter} shows $\End_\mathcal{B}(P) ^{\mathrm{op}} \cong A'$ for $\uqsl$, and the proof is identical for $D(\Lambda_n)$.  It uses the expression
\[ \End_\mathcal{B}(P) = \left( \begin{array}{ll}
    \End_\mathcal{B}(P_S) & \Hom_\mathcal{B}(P_S,P_T)
 \\ \Hom_\mathcal{B}(P_T,P_S) & \End_\mathcal{B}(P_T)
   \end{array} \right).
\]
The generators $g,x_1,x_2$ correspond to matrices of the form
\[ \left( \begin{array}{ll}
           \id & 0 \\
0& -\id
          \end{array} \right), \,\,\, \left( \begin{array}{ll}
           0 & \phi \\
0& 0
          \end{array} \right),
\,\,\, \left( \begin{array}{ll}
           0 & 0 \\
\psi& 0
          \end{array} \right)\]
respectively.
\end{proof}

Consider the subalgebra $B$ of $A$ generated by $y_1$ and $y_2$, which is isomorphic to $k\langle y_1, y_2 \rangle / (y_1 ^2, y_2 ^2, y_1y_2+y_2y_1)$.  If the characteristic of $k$ is not 2 then an $A$-module is projective if and only if it is projective on restriction to $B$.  Let $B_{\alpha \beta}=\langle \alpha y_1 +\beta y_2, g \rangle$ where $\alpha, \beta \in k$ are not both zero.
Using \cite[3.1,3.3]{BEH}, or by applying the results of Section \ref{detecting_proj_section} we see that the rank variety for $A$-modules defined by 
\[ V^r(M)= \{ (\alpha, \beta) \in k^2 : M|_{B_{\alpha \beta}} \mathrm{is\,\, not\,\, projective} \} \]
 satisfies (C1)-(C4) for all finitely generated $A$-modules. Note that with the notation of the previous section, we have $B_{\alpha \beta} \cong H_{\alpha \beta}f_+$. 

\begin{rem} The algebra $A$ is isomorphic to the quiver algebra $kQ/I$ appearing
n \cite[2.26]{EGST} and \cite{taillefer}.  If $k$ has
characteristic $2$ then $A$ is isomorphic to the group algebra of
an elementary abelian group of order $8$. 
\end{rem}

Let $e$ be an idempotent in $\mathcal{B}$ such that $P=\mathcal{B}e$, so 
$A \cong \End_\mathcal{B} (\mathcal{B} e )^\mathrm{op} \cong e\mathcal{B}e$, where the second isomorphism is given by $f \mapsto f(e)$.
Then the Morita equivalence between $_\mathcal{B}\mathrm{mod}$ and
$_{e\mathcal{B}e}\mathrm{mod}$ acts on modules by $M \mapsto eM$.
It follows that $M$ is a projective $\mathcal{B}$-module if and
only if $eM$ is projective when regarded as an $A$-module by the
isomorphism $A  \cong e\mathcal{B}e$.

\begin{defn}Let char $k \neq 2$ and let $M$ be a finitely generated $D(\Lambda_n)$-module or a $\uqsl$-module that belongs to a block $\mathcal{B}$.  Let $e$ be an idempotent in $\mathcal{B}$ such that $\mathcal{B}e$ is a progenerator.  Then the \textbf{rank variety} $V^r(M)$ is defined to be
\[  \{ (\alpha, \beta) \in k^2 : eM|_{B_{\alpha \beta}} \mathrm{is\,\, not\,\, projective} \}. \]
\end{defn}
By the previous remark this defines a rank variety satisfying
(C1)-(C4).

\begin{subsection}{Finite generation of the $\ext$-ring of $A$}
Let $k$ be a field of characteristic not 2 and let $A$ be as in (\ref{A_presentation}). Then $A$ has two projective indecomposable summands $P_+:=A e_+$ and $P_-:=A e_-$ where $e_+=1+g$ and $e_-=1-g$. We denote the respective simple modules by $k_+$ and $k_-$.
\begin{theo}
Let $M$, $N$ be two finitely generated $A$-modules. Then
$\ext^*_{A}(M,N)$ is finitely generated as
a $\ext^{*}_{A}(k,k)$-module and $\ext^{*}_{A}(k,k)$
is a finitely generated algebra.
\end{theo}
\begin{proof}
By \cite[3.3]{taillefer} or Remark \ref{ext} the cohomology ring $\ext^*_A(S,S)$ is isomorphic to 
by $k[x,y,z]/ (z^2-xy)$ where $x, y$ and $z$ are
in degree two for $S=k_-$ or $S=k_+$. Therefore  $\ext^*_A(S,S)=
\ext^{ev}_A(S,S)$ is a commutative finitely generated
ring.

By induction on the length of a composition series of $M$ and $N$, it is
sufficient to show that $\ext^{*}_A(k_+,k_-)$ is 
finitely generated as $\ext^{*}_A(k_+,k_+)$-module.  We show next that $\ext^{*}_A(k_+,k_-)$ is generated by $\ext^{*}_A(k_+,k_+) $ in degree one. 
Note that a minimal projective resolution $\mathsf{P}$ is given by 
\[ \cdots \rightarrow P_+ \oplus P_+\oplus P_+ \to P_-\oplus P_- \to P_+ \] We denote by $e^j$ the idempotent generating the $j$th copy of $P_\pm$ at a fixed degree in this projective resolution.  The differentials are then given by $\partial_i(e^{i+1})=y_{1}e^{i},$ $\partial_i(e^1)=y_2e^1$ and $\partial_i(e^j)=y_2e^{j}+y_1e^{j-1}$ for $1<j <i+1$.
Therefore $\ext^*_A(k_+,k_-)$ is non-zero only in odd degrees and
$\dim \ext^{2i+1}(k_+,k_-)=2i+2$.

Similarly a minimal projective resolution $\mathsf{P_-}$ of $k_-$ is given by 
\[ \cdots\rightarrow  P_- \oplus P_- \oplus P_- \to P_+ \oplus P_+ \to P_- \]  The differentials are then given as above.

Let $\mathcal{P}$ denote the category of projective finite-dimensional $A$-modules and let 
$K^{-,b}(\mathcal{P})$ denote the homotopy category of bounded above complexes with bounded homology.
Then we can identify $\ext^n _A(k_+, k_-)$ with $\Hom_{K^{-,b}(\mathcal{P}) }( \mathsf{P}, \mathsf{P_-} [n] )$.
So let $\zeta \in  \ext^n_A(k_+, k_-) $ be a non-zero element with $n$ odd. Then we can view $\zeta$ as an element of $\Hom_{K^{-,b}(\mathcal{P})} ( \mathsf{P} , \mathsf{P_-}[n])$. We fix $1\le t \le n+1$ and take $\zeta_t^i(e^{j})=e^{j-t+1}$ for $j=t ,\cdots, t+i-n$ and $i \ge n$ and $0$ otherwise. The elements $\zeta_t$ span $\ext^n_A(k_+,k_-)$. 

Then $\zeta_t$ factors through a map $f:\mathsf{P}[n-1] \to  \mathsf{P_-}[n] $ and $h:\mathsf{P} \to \mathsf{P}[n-1]$ as 
\[   \xymatrix{ 
\cdots \ar[r] & \oplus_{i=1}^{n+1} P_- \ar[r] \ar[d]_{h^n} & \oplus_{i=1}^{n} P_+ \ar[r] \ar[d] & \oplus_{i=1}^{n-1}P_- \ar[r] \ar[d] & \cdots \ar[r] & P_ + \ar[r] \ar[d]& 0 \\
\cdots \ar[r] & \oplus_{i=1}^2 P_- \ar[r] \ar[d]_{f^n} &P_+ \ar[r] \ar[d] & 0 \ar[d] \ar[r] & \cdots \ar[r] & 0 \ar[r] \ar[d]& 0\\
 \cdots \ar[r] & P_-  \ar[r] & 0 \ar[r] & 0 \ar[r] & \cdots \ar[r] & 0\ar[r] & 0 
} \] 
where $h^i(e^j)=e^{t-j+1}$ for $j=t, \cdots, t+ i-n+1$ and $i \ge n-1$ and zero otherwise. The map $f$ is given by $f^i(e^j)=e^{j}$ for $j=1,\cdots , i-n+1$ and $i \ge n$ and $0$ otherwise.
Then $f \in \Hom_{K^{-,b}(\mathcal{P})}(\mathsf{P}[n-1], \mathsf{P_-}[n])$ can be seen as an element of $\ext^{1}_A(k_+,k_-)$ and $h$ as an element of $\ext^{n-1}_A(k_+,k_+)$. Furthermore composition of chain maps in $K^{-,b}(\mathcal{P})$ corresponds to Yoneda multiplication in the $\ext$-ring. This finishes the proof.
\end{proof}
As the non-semi-simple indecomposable blocks of $D(\Lambda_n)$ and $\uqsl$ are Morita equivalent to $A$, we get immediately by the previous theorem:
\begin{cor}Let $H$ be $\uqsl$ or $D(\Lambda_n)$ for some $n$.
Let $M$, $N$ be finitely generated $H$-modules. Then
$\ext^*_{H}(M,N)$ is finitely generated as
$\ext^{*}_{H}(k,k)$-module and $\ext^{*}_{H}(k,k)$
is a finitely generated algebra.
\end{cor}
\begin{rem} We can use previous results to determine certain restriction maps.  Define $B_{\alpha \beta}=\langle \alpha y_1 +\beta y_2, g \rangle$ for $\alpha,\beta$ not both zero.  Then the restriction map $\res^A_{B_{\alpha \beta}}: \ext_A ^*(k,k) \to \ext_{B_{\alpha \beta}} ^* (k,k)$ is identical to the one computed in Section \ref{restriction} as the non-semisimple block $D(\Lambda_2)f_+$ of $D(\Lambda_2)$ is isomorphic to $A$, with the isomorphism sending $H_{\alpha \beta}f_+$ to $B_{\alpha \beta}$.
\end{rem}
\end{subsection}
\end{section}

\bibliography{hopf_alg_projectivity}

\providecommand{\bysame}{\leavevmode\hbox to3em{\hrulefill}\thinspace}
\providecommand{\MR}{\relax\ifhmode\unskip\space\fi MR }
\providecommand{\MRhref}[2]{%
  \href{http://www.ams.org/mathscinet-getitem?mr=#1}{#2}
}
\providecommand{\href}[2]{#2}
\begin{thebibliography}{EGST06}

\bibitem[BEH07]{BEH}
D.~Benson, K.~Erdmann, and M.~Holloway, \emph{Rank varieties for a class of
  finite-dimensional local algebras}, J. Pure Appl. Algebra \textbf{211}
  (2007), no.~2, 497--510.

\bibitem[Ben98]{bensonII}
D.~Benson, \emph{Representations and cohomology {II}: Cohomology of groups and
  modules}, 2nd ed., Cambridge studies in advanced mathematics, vol.~31,
  Cambridge University Press, 1998.

\bibitem[Car83]{carl:vars}
J.~Carlson, \emph{The varieties and the cohomology ring of a module}, J.
  Algebra \textbf{85} (1983), 104--143.

\bibitem[EGST06]{EGST}
K.~Erdmann, E.~Green, N.~Snashall, and R.~Taillefer, \emph{Representation
  theory of the {D}rinfeld doubles of a family of {H}opf algebras}, J. Pure
  Appl. Algebra \textbf{204} (2006), no.~2, 413--454.

\bibitem[FP86]{FP}
E.~Friedlander and B.~Parshall, \emph{Support varieties for restricted {L}ie
  algebras}, Invent. Math. \textbf{86} (1986), no.~3, 553--562.

\bibitem[FP07]{FPev}
E.~Friedlander and J.~Pevtsova, \emph{{$\Pi$}-supports for modules for finite
  group schemes}, Duke Math. J. \textbf{139} (2007), no.~2, 317--368.

\bibitem[Jie97]{jie}
Xiao Jie, \emph{Finite-dimensional representations of {$U\sb t({\rm sl}(2))$}
  at roots of unity}, Canad. J. Math. \textbf{49} (1997), no.~4, 772--787.

\bibitem[Kas95]{kassel}
C.~Kassel, \emph{Quantum groups}, Graduate Texts in Mathematics, vol. 155,
  Springer-Verlag, New York, 1995.

\bibitem[Pat99]{patra}
M.~K. Patra, \emph{On the structure of nonsemisimple {H}opf algebras}, J. Phys.
  A \textbf{32} (1999), no.~1, 159--166.

\bibitem[PW]{PW}
J.~Pevtsova and S.~Witherspoon, \emph{Varieties for modules of quantum
  elementary abelian groups}, to appear in Algebr. Represent. Theory, DOI
  10.1007/s10468-008-9100-y.

\bibitem[RW06]{RW}
D.~Radford and S.~Westreich, \emph{Trace-like functionals on the double of the
  {T}aft {H}opf algebra}, J. Algebra \textbf{301} (2006), no.~1, 1--34.

\bibitem[Sut94]{suter}
R.~Suter, \emph{Modules over {$\mathfrak{U}_q(\mathfrak{sl}_2)$}}, Comm. Math.
  Phys. \textbf{163} (1994), no.~2, 359--393.

\bibitem[Tai07]{taillefer}
R.~Taillefer, \emph{Bialgebra cohomology of the duals of a class of generalized
  {T}aft algebras}, Comm. Algebra \textbf{35} (2007), no.~4, 1415--1420.

\end{thebibliography}
\bibliographystyle{amsalpha}

\end{document}